\documentclass[11pt]{article} 
\usepackage{amsthm,xspace,amssymb,amsmath,amscd,amsthm,epsfig,mathrsfs,graphicx,array} 

\makeatletter
\newfont{\footsc}{cmcsc10 at 8truept}
\newfont{\footbf}{cmbx10 at 8truept}
\newfont{\footrm}{cmr10 at 10truept}
\renewcommand{\ps@plain}{%
\renewcommand{\@oddfoot}{\footsc Abacus-histories and creation operators% 
  \hfil\footrm\thepage}}
\makeatother
\theoremstyle{plain}
\newtheorem{theorem}{Theorem}

\theoremstyle{definition}

\newtheorem{example}[theorem]{Example}
\newtheorem{remark}[theorem]{Remark}

\DeclareMathOperator{\Bop}{\mathbb{B}}
\DeclareMathOperator{\Cop}{\mathbb{C}}
\DeclareMathOperator{\Hop}{\mathbb{H}} 
\DeclareMathOperator{\Sop}{\mathbb{S}}
\newcommand{\Z}{\mathbb{Z}}
\newcommand{\Q}{\mathbb{Q}}

\newcommand{\M}[1]{M_{#1}}
\newcommand{\Sst}{\mathrm{S}} % south step (\S, \SS were taken).
\newcommand{\W}{\mathrm{W}} % west step
\newcommand{\scprod}[2]{\langle #1,#2\rangle}
\newcommand{\CO}{\mathscr{C}_{\omega}} % conjugation by $\omega$
\DeclareMathOperator{\charge}{charge} % charge of SSYT
\DeclareMathOperator{\HS}{HS} % horizontal strip
\DeclareMathOperator{\id}{id} % identity map

\DeclareMathOperator{\VS}{VS} % vertical strip

\newlength{\cellsize}
\newcommand\tableau[1]{
\vcenter{
\let\\=\cr
\baselineskip=-16000pt
\lineskiplimit=16000pt
\lineskip=0pt
\halign{&\tableaucell{##}\cr#1\crcr}}}

\newcommand{\tableaucell}[1]{{%
\def \arg{#1}\def \void{}%
\ifx \void \arg
\vbox to \cellsize{\vfil \hrule width \cellsize height 0pt}%
\else
\unitlength=\cellsize
\begin{picture}(1,1)
\put(0,0){\makebox(1,1){$#1$}}
\put(0,0){\line(1,0){1}}
\put(0,1){\line(1,0){1}}
\put(0,0){\line(0,1){1}}
\put(1,0){\line(0,1){1}}
\end{picture}%
\fi}}
\begin{document} 

\title{Abacus-histories and the combinatorics of creation~operators} 
\date{\today}

\author{Nicholas A. 
Loehr\thanks{This work was supported by a grant from the Simons Foundation/SFARI
(\#633564, N.A.L.).  }
 \\ Dept. of Mathematics 
 \\ Virginia Tech 
 \\ Blacksburg, VA 24061-0123 
 \\ \texttt{nloehr@vt.edu}
\and
Gregory S. Warrington\thanks{This work was supported by a grant 
from the Simons Foundation/SFARI (\#429570, G.S.W.).  }
\\ Dept. of Mathematics and Statistics
\\ University of Vermont
\\ Burlington, VT 05401 
\\ \texttt{gregory.warrington@uvm.edu}}
% NAL: Both grant numbers confirmed. Footnotes use the exact 
% language from Simons Foundation Policies document, pg. 11.

\maketitle 
\begin{abstract}
Creation operators act on symmetric functions to build
Schur functions, Hall--Littlewood polynomials, and related symmetric
functions one row at a time. Haglund, Morse, Zabrocki, and others
have studied more general symmetric functions $H_{\alpha}$, $C_{\alpha}$, and 
$B_{\alpha}$ obtained by applying any sequence of creation operators to $1$.  
We develop new combinatorial models for the Schur expansions of these
and related symmetric functions using objects called abacus-histories. These 
formulas arise by chaining together smaller abacus-histories that encode the
effect of an individual creation operator on a given Schur function.
We give a similar treatment for operators such as multiplication
by $h_m$, $h_m^{\perp}$, $\omega$, etc., which serve as building blocks to 
construct the creation operators. We use involutions on abacus-histories to 
give bijective proofs of properties of the Bernstein creation operator
and Hall--Littlewood polynomials indexed by three-row partitions.  
\end{abstract}

% NAL: I added keywords as suggested (in journal website metadata).
% NL 7/2020: I un-commented the keywords below for the revised version.
%  (\keywords command doesn't work here.) 

% 2020 MSC Subject Classifications: 05A17, 05E05.

%%%%%%%%%%%%%%%%%%%%%%%%%%%%%%%%%%%%%%%%%%%%%%%%%%%%%%%%%%%%%%%%%%%%%%%%%%%%%% 
\section{Introduction}
\label{sec:intro}

Creation operators are an important technical tool in the study
of the Schur polynomials $s_{\mu}$, 
the Hall--Littlewood polynomials $H_{\mu}$, and related symmetric functions.
Let $\Lambda$ denote the ring of symmetric functions with coefficients
in the field $F=\Q(q)$, where $q$ is a formal variable.
For each integer $b$, the \emph{Bernstein creation operator}
$\Sop_b$ is an $F$-linear operator on $\Lambda$. These operators ``create''
the Schur symmetric functions, one row at a time, in the following sense.
Given any integer partition $\mu=(\mu_1\geq\mu_2\geq\cdots\geq \mu_L)$,
we have
\begin{equation}\label{eq:s-create}
 s_{\mu} =\Sop_{\mu_1}\circ \Sop_{\mu_2}\circ\cdots\circ \Sop_{\mu_L}(1). 
\end{equation}
Similarly, the \emph{Jing creation operators}~\cite{Jing91} are $F$-linear 
operators $\Hop_b$ on $\Lambda$ that create the Hall--Littlewood symmetric
functions $H_{\mu}$~\cite[Chpt. III]{Mac}. 
Specifically, for any integer partition $\mu$,
\begin{equation}\label{eq:HL-create}
 H_{\mu} =\Hop_{\mu_1}\circ \Hop_{\mu_2}\circ\cdots\circ \Hop_{\mu_L}(1). 
\end{equation}

Garsia, Haglund, Morse, Xin, and Zabrocki~\cite{GXZ11,HMZ12} 
have studied variations
of the Jing creation operators, denoted $\Cop_b$ and $\Bop_b$, that play
a crucial role in the study of $q,t$-Catalan numbers, diagonal harmonics
modules, and the Bergeron--Garsia nabla operator. Replacing each
$\Hop_{\mu_i}$ in~\eqref{eq:HL-create} by $\Cop_{\mu_i}$ or $\Bop_{\mu_i}$
produces symmetric functions that are closely related to Hall--Littlewood
polynomials. More generally, we can consider operators indexed by
arbitrary compositions rather than restricting to partitions. 
For any finite sequence of integers
$\alpha=(\alpha_1,\alpha_2,\ldots,\alpha_L)$, we can define symmetric
functions
\begin{align}
\label{eq:Salpha}
 S_{\alpha}&=\Sop_{\alpha_1}\circ\Sop_{\alpha_2}\circ\cdots\circ\Sop_{\alpha_L}(1);
\\ \label{eq:Halpha}
H_{\alpha}&=\Hop_{\alpha_1}\circ\Hop_{\alpha_2}\circ\cdots\circ\Hop_{\alpha_L}(1);
\\ \label{eq:Calpha}
C_{\alpha}&=\Cop_{\alpha_1}\circ\Cop_{\alpha_2}\circ\cdots\circ\Cop_{\alpha_L}(1);
\\ \label{eq:Balpha}
B_{\alpha}&=\Bop_{\alpha_L}\circ\Bop_{\alpha_{L-1}}\circ\cdots\circ
\Bop_{\alpha_1}(1).
\end{align}
On one hand, as we explain in Section~\ref{subsec:defineS},
each $S_{\alpha}$ is either $0$ or $\pm s_{\mu}$ for some
partition $\mu$, where $\mu$ can be found from $\alpha$ by repeated use
of the \emph{commutation rule}
\begin{equation}\label{eq:S-commute}
 \Sop_m\circ \Sop_n=-\Sop_{n-1}\circ\Sop_{m+1}\qquad(m,n\in\Z).
\end{equation} 
On the other hand, 
the $H_{\alpha}$, $C_{\alpha}$, and $B_{\alpha}$ are more complicated
symmetric functions that may be regarded as generalized
Hall--Littlewood polynomials. For general $\alpha$, the Schur coefficients
of these symmetric functions are polynomials in $q$ 
(possibly multiplied by a fixed negative power of $q$)
containing a mixture of positive and negative coefficients. 

The primary goal of this paper is to develop explicit combinatorial
formulas for the Schur expansions of $H_{\alpha}$, $C_{\alpha}$,
and $B_{\alpha}$ based on signed, weighted collections of non-intersecting
lattice paths called \emph{abacus-histories}. Along the way, we develop
concrete formulas giving the Schur expansion of the image of any Schur
function under a single operator $\Sop_b$, $\Hop_b$, $\Cop_b$, $\Bop_b$, 
or any finite composition of such operators. We also give a similar 
treatment for some simpler operators such as $\omega$, multiplication by
$h_b$, $h_b^{\perp}$, etc., which serve as building blocks for constructing
the more elaborate creation operators.  

Some related work appears 
in a paper by Jeff Remmel and Meesue Yoo~\cite{RemYoo}. Our
% NL 7/2020: omitted "dramatically simpler" (ref2,#1).
approach features two key innovations leading to new and explicit 
combinatorial formulas. First, we use \emph{abacus diagrams} rather
% NL 7/2020: changed "partition diagrams" to "Ferrers diagrams" (ref2,#3).
than \emph{Ferrers diagrams} as a means of visualizing the indexing
partition $\mu$ for a Schur function $s_{\mu}$. This lets us record
a particular Schur coefficient using a one-dimensional picture instead
of a two-dimensional picture. Second, we utilize the second dimension
of our picture to show the \emph{evolution of the abacus over time}
as various operators are applied to our initial Schur function. We thereby
generate collections of non-intersecting lattice paths (abacus-histories)
that represent the individual terms in the Schur expansion of the
desired symmetric function.  In some instances,
we can define involutions on abacus-histories that cancel out negative
objects, thereby proving Schur-positivity or related identities.

The rest of this paper is organized as follows.
Section~\ref{sec:alg-create} reviews the needed background on symmetric 
functions and covers definitions and algebraic properties
of various creation operators.  
Section~\ref{sec:comb-create} develops combinatorial versions of the creation
operators, showing how to implement each operator by acting on an abacus
diagram for one or two time steps. We use this combinatorics to
reprove (from first principles) some creation operator identities 
such as~\eqref{eq:s-create} and~\eqref{eq:S-commute}.  
Section~\ref{sec:BCH} iterates our constructions to develop abacus-history
formulas for the Schur expansions of $H_{\alpha}$, $C_{\alpha}$, $B_{\alpha}$,
and related symmetric functions.
As a sample application of this technology,
Section~\ref{subsec:HLpoly} gives an elementary combinatorial proof of 
the Schur-positivity of three-row Hall--Littlewood polynomials,
leading to a simple formula for the Schur coefficients of these objects.
We conclude by presenting some open problems and directions for further work.

\section{Algebraic Development of Creation Operators}
\label{sec:alg-create}

We assume readers are familiar with basic background material on symmetric
functions, including definitions and facts concerning integer partitions,
the elementary symmetric functions $e_k$,
the complete homogeneous symmetric functions $h_k$,
the Schur symmetric functions $s_{\mu}$,
the involution $\omega$, and the Hall scalar product $\scprod{\cdot}{\cdot}$
on $\Lambda$.  In particular, the Schur functions $s_{\mu}$ (with $\mu$ 
ranging over all integer partitions) form an orthonormal basis of $\Lambda$ 
relative to the Hall scalar product, and $\omega$ is an involution,
ring isomorphism, and isometry on $\Lambda$ sending each 
$s_{\mu}$ to $s_{\mu'}$.  (See standard texts on symmetric functions
such as~\cite{loehr-comb,Mac,sagan} for more information.)

% NL 7/2020: changed "plethysm" to "plethystic formula" or
%  "plethystic notation" here and below (ref2,#2).
Our initial definitions of the creation operators 
(following~\cite{GXZ11,HMZ12})
utilize \emph{plethystic notation}, but readers need not have any
detailed prior knowledge of plethystic notation to understand this paper.
In fact, one of our goals here is to offer an alternative, highly
concrete and combinatorial treatment of creation operators to complement the
% NL 7/2020: removed "intimidating" and "dominate" on next line (ref2,#1).
plethystic computations that appear in much of the existing
literature on this topic. Thus, each plethystic definition is immediately
followed by an equivalent algebraic definition not using plethystic
notation.  Familiarity with plethystic notation is required in only 
one (optional) section
that proves the equivalence of the two definitions. The paper~\cite{expose}
has a detailed exposition of plethystic notation containing all facts needed 
here.  Later in the paper, we develop completely combinatorial definitions
of creation operators (and related operators) in terms of abacus-histories.

\subsection{Multiplication Operators and their Adjoints}
\label{subsec:mult-adjoint}

Recall that $F$ is the field $\Q(q)$, and $\Lambda$ is the $F$-algebra
of symmetric functions with coefficients in $F$.
For any symmetric function $f\in\Lambda$, define the linear operator 
$\M{f}:\Lambda\rightarrow\Lambda$ to be \emph{multiplication by $f$}:
\begin{equation}\label{eq:mult-by-f}
 \M{f}(P)=fP \qquad(P\in\Lambda).
\end{equation}
We frequently take $f$ to be
$h_c$ (the complete homogeneous symmetric function) or
$e_c$ (the elementary symmetric function). 

The Pieri Rules~\cite[Sec. 9.11]{loehr-comb}
 show how $\M{h_c}$ and $\M{e_c}$ act on the Schur basis. 
First, let $\HS(c)$ be the set of all skew shapes $\lambda/\mu$
consisting of a horizontal strip of $c$ cells. For all partitions $\mu$,
\begin{equation}\label{eq:pieri-h}
 \M{h_c}(s_{\mu})=h_cs_{\mu} 
=\sum_{\lambda:\ \lambda/\mu\in\HS(c)} s_{\lambda}.
\end{equation}
Pictorially, we apply $\M{h_c}$ to $s_{\mu}$
by adding a horizontal strip of size $c$ to the Ferrers diagram of $\mu$
in all possible ways and summing the Schur functions indexed by the
new diagrams.
% NL 7/2020: changed "partition diagrams" -> "diagrams" on line above.

Second, let $\VS(c)$ be the set of all skew shapes $\lambda/\mu$
consisting of a vertical strip of $c$ cells. For all partitions $\mu$,
\begin{equation}\label{eq:pieri-e}
 \M{e_c}(s_{\mu})=e_cs_{\mu} 
=\sum_{\lambda:\ \lambda/\mu\in\VS(c)} s_{\lambda}.
\end{equation}
This time, we compute $\M{e_c}(s_{\mu})$ 
by adding a vertical strip of size $c$ to the diagram of $\mu$
in all possible ways and summing the resulting Schur functions.

For any linear operator $G$ on $\Lambda$, let $G^{\perp}$ denote
the operator on $\Lambda$ that is \emph{adjoint} to $G$ relative
to the Hall scalar product. So, $G^{\perp}$ is the unique linear map
on $\Lambda$ satisfying 
\begin{equation}\label{eq:adjoint-def}
 \scprod{G^{\perp}(P)}{Q}=\scprod{P}{G(Q)}
\quad\mbox{ for all $P,Q\in\Lambda$.}
\end{equation}
When $G$ is a multiplication operator $\M{f}$, we define
$f^{\perp}=(\M{f})^{\perp}$ for brevity. Thus,
\begin{equation}\label{eq:fperp}
 \scprod{f^{\perp}(P)}{Q}=\scprod{P}{fQ}
\quad\mbox{ for all $f,P,Q\in\Lambda$.} 
\end{equation} 
We mostly use $h_c^{\perp}$ and $e_c^{\perp}$ acting on the Schur basis.
Since the Schur basis is orthonormal relative to the Hall scalar product,
it follows from~\eqref{eq:fperp} and~\eqref{eq:pieri-h} that
%[GIVE MORE DETAILS?] 
% Details appear in comment below formula if we need them later.
\begin{equation}\label{eq:hperp}
 h_c^{\perp}(s_{\mu})=\sum_{\nu:\ \mu/\nu\in\HS(c)} s_{\nu}.
\end{equation}
%[More on how to check this: for any $Q$ of the form $s_{\nu}$,
% the scalar product of $s_{\mu}$ with $h_cs_{\nu}$ is $1$ 
% if $\mu$ arises from $\nu$ by adding a horizontal $c$-strip and 
% $0$ otherwise. Taking the scalar product of the right side
% of~\eqref{hperp} with $s_{\nu}$ gives the same result.] 
%
In other words, $h_c^{\perp}$ acts on $s_{\mu}$ by \emph{removing}
a horizontal $c$-strip from the Ferrers diagram of $\mu$ in all possible ways
and summing the resulting Schur functions. Similarly,
\begin{equation}\label{eq:eperp}
 e_c^{\perp}(s_{\mu})=\sum_{\nu:\ \mu/\nu\in\VS(c)} s_{\nu}.
\end{equation} 
So $e_c^{\perp}$ acts on $s_{\mu}$ by removing
a vertical $c$-strip from the Ferrers diagram of $\mu$ in all possible ways
and summing the resulting Schur functions. 

As a convention, when $c$ is a negative integer, we define
the operators $\M{h_c}$, $\M{e_c}$, $h_c^{\perp}$, and $e_c^{\perp}$
to be the zero operator.

\subsection{Bernstein's Creation Operators $\Sop_m$}
\label{subsec:defineS} 

As in~\cite[pg. 834]{HMZ12},
% NL 7/2020: added three citations to Thibon papers (ref2,#5).
we give a plethystic formula defining the Bernstein creation operators 
$\Sop_m$.  More information on these operators (which can be combined
into a single operator denoted $\Sop$ or $\Gamma(z)$) 
appears in earlier works by Thibon et al.~\cite{SchThi,SchThiWyb,Thibon}.
For any integer $m$ and any symmetric function $P$, set
\begin{equation}\label{eq:defS-pleth}
 \Sop_m(P)=\left.\left\{P\left[X-\frac{1}{z}\right]\sum_{k=0}^{\infty}
 h_kz^k\right\}\right|_{z^m}. 
\end{equation}
To explain this formula briefly: we first compute $P[X-(1/z)]$ by expressing
$P$ (uniquely) as a polynomial in the power-sum symmetric functions $p_n$
with coefficients in $F$, then replacing each $p_n$ by $p_n-(1/z^n)$.
Next we multiply by the formal power series $\sum_{k\geq 0} h_kz^k$
to obtain a formal Laurent series in $z$ with coefficients in $\Lambda$.
Taking the coefficient of $z^m$ in this series gives us $\Sop_m(P)$.

Here is an equivalent algebraic definition of $\Sop_m$ not using 
plethystic notation:
\begin{equation}\label{eq:defS-alg}
 \Sop_m=\sum_{c=0}^{\infty} (-1)^c \M{h_{m+c}}\circ e_c^{\perp}.
\end{equation}
(This definition appears in~\cite[Ex. 29, pp. 95--97]{Mac},
but Macdonald uses the notation $B_m$ for our $\Sop_m$. 
We use the notation~$\Sop_m$ from~\cite{HMZ12}
to avoid confusion with the operator $\Bop_m$ below.)
We prove the equivalence of definitions~\eqref{eq:defS-pleth} 
and~\eqref{eq:defS-alg} in Section~\ref{sec:pleth-prf}.

By combining the Pieri formulas and dual 
Pieri formulas discussed above, we can give a combinatorial prescription
% NL 7/2020: changed "partition" to "Ferrers" on next line (ref2,#3).
for computing $\Sop_m(s_{\mu})$ based on Ferrers diagrams.
Starting with the diagram of $\mu$, do the following steps in all
possible ways. First choose a nonnegative integer $c$.
Then remove a vertical strip of $c$ cells from $\mu$ to get some shape
$\nu$. Then add a horizontal strip of $m+c$ cells to $\nu$ to
get a new shape $\lambda$. Record $(-1)^c s_{\lambda}$ as one
of the terms in $\Sop_m(s_{\mu})$.

Now, there is a much simpler way of computing
$\Sop_m(s_{\mu})$ based on formulas~\eqref{eq:s-create},
\eqref{eq:Salpha}, and~\eqref{eq:S-commute}. Given a partition 
$\mu=(\mu_1\geq\cdots\geq\mu_L)$ and integer $m$, start 
with the list $(m,\mu_1,\ldots,\mu_L)$. If $m\geq\mu_1$,
then output the Schur function indexed by this new list.
Otherwise, repeatedly perform the following steps on the list
% NAL: In response to Greg's comment, I changed the description of the sign.
(with infinitely many zero parts appended). Initialize a sign variable
$\epsilon=+1$.  Look for the unique ascent $a<b$ in the current list.
If $b=a+1$, then return zero as the answer.  
Otherwise replace the sublist $(a,b)$ by $(b-1,a+1)$, 
replace $\epsilon$ by $-\epsilon$, and continue.  
We eventually return zero or obtain a weakly decreasing
list of nonnegative integers. In the latter case, return $\epsilon$
times the Schur function indexed by this list. 

% NL 7/2020: Added this paragraph on Jacobi--Trudi determinants (ref2,#4).
This algorithm is a version of Littlewood's method for straightening
Jacobi--Trudi determinants~\cite{littlewood}.  Given any list of integers
$\alpha=(\alpha_1,\ldots,\alpha_L)$, let $D(\alpha)$ be the
determinant of the $L\times L$ matrix with $i,j$-entry $h_{\alpha_i+j-i}$.
For an integer partition $\mu$, $D(\mu)$ is the Schur function
$s_{\mu}$ by the Jacobi--Trudi formula. For any $\alpha$,
$D(\alpha)$ is either $0$ or $\pm s_{\nu}$ for some partition $\nu$. 
We can find $\nu$ by repeatedly interchanging rows $i$ and $i+1$ of
the matrix where $\alpha_i<\alpha_{i+1}$. 
Each such interchange causes a sign change and replaces 
parts $\alpha_i$ and $\alpha_{i+1}$ in $\alpha$ by $\alpha_{i+1}-1$
and $\alpha_i+1$, respectively. Comparing to the previous paragraph,
we see that $\Sop_m(s_{\mu})$ is none other than $D(m,\mu_1,\ldots,\mu_L)$.
The straightening process can also be performed visually on
composition diagrams using the \emph{slinky rule}, as illustrated
in~\cite{slinky}.

\begin{example}\label{ex:computeS}
Given $\mu=(8,8,8,4,4,2,2,1)$ and $m=-2$, we compute
\begin{align*}
 \Sop_{-2}(s_{\mu})
 &=S_{(-2,8,8,8,4,4,2,2,1)}
 =-S_{(7,-1,8,8,4,4,2,2,1)}
 =+S_{(7, 7,0,8,4,4,2,2,1)}
\\& =-S_{(7, 7,7,1,4,4,2,2,1)}
 =+S_{(7, 7,7,3,2,4,2,2,1)}
 =-S_{(7, 7,7,3,3,3,2,2,1)}
\\& =-s_{(777333221)}. 
\end{align*}
By similar calculations, we find the values of $\Sop_m(s_{\mu})$ shown in
Table~\ref{tab:computeS}.
%As further examples, $\Sop_m(s_{\mu})=0$ for $m=5,6,7$; 
%$\Sop_3(s_{\mu})=-s_{(777644221)}$;
%$\Sop_{-8}(s_{\mu})=+s_{(777331100)}$; and 
%$\Sop_m(s_{\mu})=s_{(m,\mu)}$ for all $m\geq 8$.
\end{example}

\begin{table}
\begin{center}
% NAL 7/2020: Note the array environment is used here. 
\[ \begin{array}{lll}
\Sop_m(s_{\mu}) = +s_{(m,88844221)} \mbox{ for all $m\geq 8$,}
& \Sop_4(s_{\mu}) = -s_{(777744221)},
& \Sop_3(s_{\mu}) = -s_{(777644221)},
\\ \Sop_2(s_{\mu}) = -s_{(777544221)},
& \Sop_1(s_{\mu}) = -s_{(777444221)},
& \Sop_{-2}(s_{\mu}) = -s_{(777333221)},
\\ \Sop_{-3}(s_{\mu}) = -s_{(777332221)},
& \Sop_{-6}(s_{\mu}) = -s_{(777331111)},
& \Sop_{-8}(s_{\mu}) = +s_{(777331100)},
\\ \Sop_{m}(s_{\mu}) = 0 \mbox{ for all other $m\in\Z$.} &&
\end{array} \]
\caption{Values of $\Sop_m(s_{(88844221)})$ for all $m\in\Z$.}
\label{tab:computeS}
\end{center}
\end{table}

It is not obvious that the two methods we have described for computing 
$\Sop_m(s_{\mu})$ always give the same result. 
We prove this fact later using abacus-histories, 
and we will also give direct combinatorial proofs of~\eqref{eq:s-create}
and~\eqref{eq:S-commute}.

\subsection{Jing's Creation Operators $\Hop_m$}
\label{subsec:defineH}

Here is a plethystic definition of Jing's creation operators $\Hop_m$.
As in~\cite[(2.2)]{HMZ12}, for all $m\in\Z$ and $P\in\Lambda$, let
\begin{equation}\label{eq:defH-pleth}
 \Hop_m(P)=\left.\left\{ P\left[X+\frac{q-1}{z}\right]
     \sum_{k=0}^{\infty} h_kz^k\right\}\right|_{z^m}. 
\end{equation}
In this case, $P[X+(q-1)/z]$ denotes the image of $P$ under the
plethystic substitution sending each $p_n$ to $p_n+(q^n-1)/z^n$.

Alternatively, we could define
\begin{equation}\label{eq:defH-alg}
 \Hop_m=\sum_{c\geq 0} q^c\Sop_{m+c}\circ\, h_c^{\perp}.  
\end{equation}
We prove the equivalence of definitions~\eqref{eq:defH-pleth}
and~\eqref{eq:defH-alg} in Section~\ref{sec:pleth-prf}.

% NL 7/2020: changed "partition" to "Ferrers" on next line (ref2,#3).
We can compute $\Hop_m(s_{\mu})$ via Ferrers diagrams as follows.
Starting with the diagram of $\mu$, do the following steps in all possible
ways. First choose a nonnegative integer $c$. Then remove a horizontal
strip of $c$ cells from $\mu$ to get some shape $\nu$. Compute
$\Sop_{m+c}(\nu)$ as described earlier to obtain zero or a signed Schur 
function $\pm s_{\lambda}$. Record $q^c$ times the answer
as one of the terms in the Schur expansion of $\Hop_m(s_{\mu})$. By iterating
this description, it is evident that for every list of integers $\alpha$,
the Schur coefficients of $H_{\alpha}$ are polynomials in $q$ with
integer coefficients. Jing~\cite{Jing91} proved that when $\alpha$ is a 
partition $\mu$, $H_{\mu}=\Hop_{\mu_1}\circ\cdots\circ\Hop_{\mu_L}(1)$ is none
other than the Hall--Littlewood symmetric function indexed by $\mu$.

\subsection{The Creation Operators $\Cop_m$ and $\Bop_m$}
\label{subsec:defineCB}

We use~\cite[Remark 3.7, pg. 835]{HMZ12} as our plethystic definition of
the creation operator $\Cop_m$. For $m\in\Z$ and $P\in\Lambda$, let
\begin{equation}\label{eq:defC-pleth}
 \Cop_m(P)=\left.\left\{(-q^{-1})^{m-1}P\left[X+\frac{q^{-1}-1}{z}\right]
     \sum_{k=0}^{\infty} h_kz^k\right\}\right|_{z^m}. 
\end{equation}
This operator is a variation of $\Hop_m$ obtained by replacing
$q$ by $1/q$ in $\Hop_m$, and then multiplying the output 
by a global factor $(-1/q)^{m-1}$. So~\eqref{eq:defH-alg} leads at once
to the following alternative definition of $\Cop_m$:
\begin{equation}\label{eq:defC-alg}
 \Cop_m=(-q^{-1})^{m-1}\sum_{c\geq 0} q^{-c}\Sop_{m+c}\circ\, h_c^{\perp}.  
\end{equation}

Proposition~3.6 of~\cite{HMZ12} proves an inverse version of this
identity, namely
\[ \Sop_m=(-q)^{m-1}\sum_{i\geq 0} \Cop_{m+i}\circ\, e_i^{\perp}. \]
Creation operators satisfy some useful commutation relations.
For example, Proposition 3.2 of~\cite{HMZ12} shows that for $m,n\in\Z$,
\[ q\Cop_m\circ\Cop_n-\Cop_{m+1}\circ\Cop_{n-1}
  =\Cop_n\circ\Cop_m-q\Cop_{n-1}\circ\Cop_{m+1}, \]
and in particular $q\Cop_m\circ\Cop_{m+1}=\Cop_{m+1}\circ\Cop_m$.
Analogous relations for $\Hop_m$ were proved much earlier by
Jing (see~(1.1) in~\cite{Jing91} or~(0.18) in~\cite{Jing-thesis}).

Finally, we define the creation operator $\Bop_m$ by conjugating
$\Hop_m$ by $\omega$ (see~\cite[pg. 829]{HMZ12}):
\begin{equation}\label{eq:defB}
 \Bop_m=\omega\circ\Hop_m\circ\, \omega.  
\end{equation}
Recall that $\omega$ is the linear operator on $\Lambda$ sending
each Schur function $s_{\lambda}$ to $s_{\lambda'}$,  
where $\lambda'$ is the partition conjugate to $\lambda$
obtained by transposing the Ferrers diagram of $\lambda$.
Proposition 3.5 of~\cite{HMZ12} shows that for $m+n>0$, 
$\Bop_n\circ\Cop_m=q\Cop_m\circ\Bop_n$.

\subsection{Algebraic Rules for Conjugation by $\omega$}
\label{subsec:alg-conj-omega}

Let $\CO$ denote conjugation by $\omega$, which sends any
operator $G$ on $\Lambda$ to $\CO(G)=\omega\circ G\circ \omega$.  
We now give some useful identities involving $\CO$.  First, 
\begin{equation}\label{eq:conjMf}
\CO(\M{f})=\M{\omega(f)}\quad\mbox{ for all $f\in\Lambda$.}
\end{equation}
To check this, recall that $\omega$ is a ring homomorphism on
$\Lambda$ and an involution ($\omega\circ \omega=\id$). So for 
any $P\in\Lambda$,
\[ \CO(\M{f})(P)=\omega(\M{f}(\omega(P)))
   =\omega(f\cdot\omega(P))=\omega(f)\cdot\omega(\omega(P))
   =\omega(f)\cdot P=\M{\omega(f)}(P). \]

Second, $\omega^{\perp}=\omega$. This follows since 
$\omega$ is an involution and an isometry (relative to the Hall
scalar product), which means that for all $P,Q\in\Lambda$,
$\scprod{\omega(P)}{Q}=\scprod{P}{\omega(Q)}=\scprod{\omega^{\perp}(P)}{Q}$.

Third,
\begin{equation}\label{eq:conjperp}
 \CO(f^{\perp})=(\omega(f))^{\perp}\quad\mbox{ for all $f\in\Lambda$.} 
\end{equation}
To see this, we use the first two facts and the adjoint property
 $(F\circ G)^{\perp}=G^{\perp}\circ F^{\perp}$ to compute:
\[ \CO(f^{\perp})=\omega\circ(\M{f})^{\perp}\circ \omega
  =\omega^{\perp}\circ (\M{f})^{\perp}\circ\omega^{\perp}
  =(\omega\circ\M{f}\circ\omega)^{\perp}=(\M{\omega(f)})^{\perp}
  =(\omega(f))^{\perp}. \]

Fourth, using~\eqref{eq:defS-alg} and $\CO(F\circ G)=\CO(F)\circ\CO(G)$,
we find that
\[ \CO(\Sop_m)=\sum_{c\geq 0} (-1)^c \CO(\M{h_{m+c}})\circ\CO(e_c^{\perp})
            =\sum_{c\geq 0} (-1)^c \M{e_{m+c}}\circ h_c^{\perp}. \]

Fifth, using this result and~\eqref{eq:defH-alg}, we get
\[ \Bop_m =\CO(\Hop_m)=\sum_{d\geq 0} q^d\CO(\Sop_{m+d})\circ\CO(h_d^{\perp})
 =\sum_{c\geq 0}\sum_{d\geq 0}
      q^d(-1)^c \M{e_{m+d+c}}\circ h_c^{\perp}\circ e_d^{\perp}. \] 
% NL 7/2020: changed "partition" to "Ferrers" on next line (ref2,#3).
So, we can compute $\Bop_m(s_{\mu})$ via Ferrers diagrams as follows.
Starting with the diagram of $\mu$, do the following steps in all
possible ways. First, choose integers $c,d\geq 0$. Remove a 
vertical strip of $d$ cells from $\mu$, then remove a horizontal strip
of $c$ cells, then add a vertical strip of $m+d+c$ cells. Record 
$q^d(-1)^c$ times the Schur function indexed by the new shape as
one of the terms in the Schur expansion of $\Bop_m(s_{\mu})$. Later,
we use abacus-histories to find a more efficient combinatorial rule
for computing this Schur expansion.

%\[ \Sop_m^{\perp}=\sum_{c\geq 0} (-1)^c \M{e_c}\circ h_{m+c}^{\perp}. \]

\section{Combinatorial Development of Creation Operators}
\label{sec:comb-create}

This section develops combinatorial formulas for the Schur expansions
of $G(s_{\mu})$, where $s_{\mu}$ is any Schur function and $G$ is one
of the operators $\M{h_m}$, $h_m^{\perp}$, $\M{e_m}$, $e_m^{\perp}$, $\omega$,
$\Sop_m$, $\Hop_m$, $\Cop_m$, or $\Bop_m$. These formulas are based
on the \emph{abacus model} for representing integer partitions. 
James and Kerber~\cite{JK} introduced abaci to prove facts about $k$-cores
and $k$-quotients of integer partitions. Abaci with labeled beads can
be used to prove many fundamental facts about Schur functions, including
the Pieri Rules for expanding $s_{\mu}h_k$ and $s_{\mu}e_k$ and the
Littlewood--Richardson Rule~\cite{loehr-abacus}. In our work here, it
suffices to consider unlabeled abaci. We introduce the new ingredient
of tracking the evolution of the abacus over time to model compositions
of operators applied to a given Schur function. This leads to 
new combinatorial objects called \emph{abacus-histories} that
model the Schur expansions of $H_{\alpha}$, $C_{\alpha}$, $B_{\alpha}$,
and other symmetric functions built by composing creation operators.

\subsection{The Abacus Model}
\label{subsec:abacus-model}

% NL 7/2020: I edited this subsection quite a bit to address
%  ref2,#7#8#9. Hopefully the new version is clearer.
First we review the correspondence between partitions and abaci.
Suppose $N>0$ is fixed and $\mu=(\mu_1,\mu_2,\ldots,\mu_N)$ is an integer 
partition (weakly decreasing sequence) consisting of $N$ nonnegative parts.
Let $\delta(N)=(N-1,N-2,\ldots,2,1,0)$. The map sending $\mu$ to
$\mu+\delta(N)=(\mu_1+N-1,\mu_2+N-2,\ldots,\mu_N)$ is a bijection
from the set of weakly decreasing sequences of $N$ nonnegative integers
onto the set of strictly decreasing sequences of $N$ nonnegative integers.
We visualize the sequence $\mu+\delta(N)$ by drawing
an abacus with positions numbered $0,1,2,\ldots$, and placing
a bead in position $\mu_i+N-i$ for $1\leq i\leq N$. We use $\mu+\delta(N)$
rather than $\mu$ because each position can hold at most one bead.
To formalize this concept, we define an \emph{$N$-bead abacus} to be 
a word $w=w_0w_1w_2\cdots$ with all $w_i\in\{0,1\}$ and $w_i=1$ for exactly
$N$ indices $i$. Here $w_i=1$ means the abacus has a bead in position $i$,
while $w_i=0$ means the abacus has a gap in position $i$. For example,
if $N=10$ and $\mu=(8,8,8,4,4,2,2,1,0,0)$, the associated abacus is
\begin{equation}\label{eq:abacus-ex}
 w=11010110011000011100000\cdots. 
\end{equation}

In the theory of symmetric functions, we usually identify 
two partitions that differ only by adding or deleting zero parts.
In fact, it is often most convenient to regard an integer partition
as an infinite weakly decreasing sequence ending in infinitely many zeroes.
To model such a sequence $\mu$ as an abacus, we use a doubly-infinite
%  NL 7/2020: replaced this paragraph for ref2,#8,#9.
%The $N$-bead abacus for $\mu$ gives a way to visualize the
%Schur \emph{polynomial} $s_{\mu}(x_1,x_2,\ldots,x_N)$
%in $N$ variables. Similarly, we can visualize the Schur 
%\emph{function} $s_{\mu}(x_1,x_2,\ldots)$ by using an abacus with
%infinitely many beads. In this case, the abacus is a doubly-infinite
word $w=(w_i:i\in\Z)$ such that $w_i=1$ for all $i<0$, $w_0=0$,
and $w_i=1$ for only finitely many indices $i\geq 0$. The nonzero parts of
$\mu$ can be recovered from the abacus $w$ by counting the number of gaps
to the left of each bead on the positive side of the abacus.
The convention of putting the leftmost gap at position $0$ is not essential;
any left-shift or right-shift of the word $w$ leads to the same combinatorics.

We can also construct the abacus for a partition $\mu$ by following 
the frontier of the Ferrers diagram of $\mu$. As illustrated in 
Figure~\ref{fig:ptn-abacus}, the frontier consists of 
infinitely many north steps (corresponding to the infinitely many
zero parts at the end of the sequence $\mu$), followed by a sequence of
east and north steps that follow the edge of the diagram, followed
by infinitely many east steps at the top edge. Replacing each north
step with a bead, replacing each east step with a gap, and declaring
the first east step to have index $0$ produces the doubly-infinite abacus
associated with $\mu$. The singly-infinite abacus $w$ in~\eqref{eq:abacus-ex}
is the ten-bead version of the full abacus shown in 
Figure~\ref{fig:ptn-abacus}, obtained by discarding all 
beads to the left of the tenth bead from the right
and shifting the origin to this location.

For convenience, we mostly use $N$-bead abaci in this paper, 
which leads to identities valid for symmetric polynomials
in $N$ variables.  However, for certain abacus moves to work, we must be sure 
to pad $\mu$ with enough zero parts (corresponding to beads at the 
far left end of the abacus) since these beads might participate in the move.
This reflects the algebraic fact that some symmetric function identities
are only true provided the number of variables is large enough.

\begin{figure}
\begin{center}
\epsfig{file=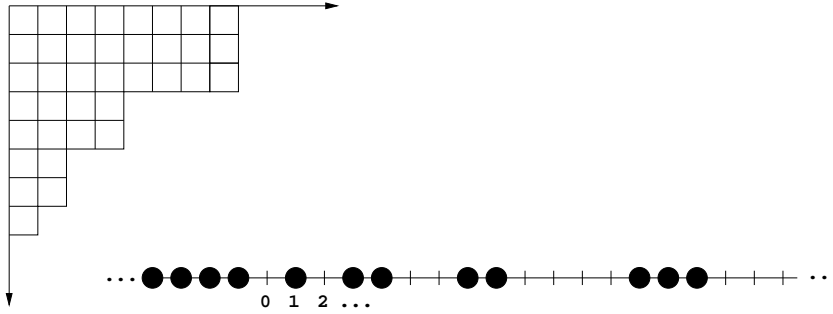,scale=0.6}
% NL 7/2020: changed "partition" to "Ferrers" on next line (ref2,#3).
\caption{The Ferrers diagram of $\mu$ and the doubly-infinite abacus
 built from the frontier of $\mu$.}
\label{fig:ptn-abacus}
\end{center}
\end{figure}

\subsection{Abacus Versions of $\M{h_m}$ and $h_m^{\perp}$}
\label{subsec:abacus-h}

Now we describe how to compute $G(s_{\mu})$ using abaci, for various
operators $G$.  Here and below, we represent the input $s_{\mu}$ 
as an $N$-bead abacus drawn in the top row of a diagram. Each operator
$G$ acts by moving beads on the abacus according to certain rules,
producing several possible new abaci that may be multiplied by signs
or weights (powers of $q$). Each abacus stands for the Schur function
indexed by the partition corresponding to the abacus.  
We make a diagram for each possible new abacus produced by $G$,
where the output abacus appears in the second row (see the figures
below for examples).  Moving downward through successive rows represents
the passage of time as various operators are applied to the initial abacus.  
This convention lets us use a two-dimensional picture to display the
evolution of an abacus as a whole sequence of operators are performed.
It is much more difficult to visualize such an operator sequence
% NL 7/2020: changed "partition" to "Ferrers" on next line (ref2,#3).
using Ferrers diagrams, particularly when some operators act by
adding cells and others act by removing cells.

As our first example, consider the computation of 
$\M{h_m}(s_{\mu})=h_ms_{\mu}$ using abaci. 
By the Pieri Rule~\eqref{eq:pieri-h}, we know $h_ms_{\mu}$
is the sum of all $s_{\lambda}$ where $\lambda$ is obtained by adding 
a horizontal strip of $m$ cells to the diagram of $\mu$ in all possible ways.
Using the correspondence between the frontier of $\mu$ and the
abacus for $\mu$, it is routine to check that adding such a horizontal
strip corresponds to moving various beads right a total of $m$ positions 
on the abacus.  (See~\cite{loehr-abacus} or~\cite[Sec. 10.10]{loehr-comb} 
for more details.) A given bead may move more than once, but no bead may move 
into a position originally occupied by another bead.

To record this move in an abacus-history diagram, we start in row $1$ 
with an $N$-bead abacus for $\mu$, where $\mu$ must end in at least one
part equal to $0$. We draw the beads as dots located at lattice points
in row $1$. Next, in all possible ways, we draw a total of $m$ east
steps starting at various beads, then move every bead $1$ step south
to represent the passage of time. For example, Figure~\ref{fig:mult-h}
shows the abacus-histories encoding the computation
\[ \M{h_2}(s_{(3110)})=s_{(5110)}+s_{(4210)}+s_{(4111)}
   +s_{(3310)}+s_{(3211)}. \]
Note that the extra zero part is needed to accommodate horizontal strips
that use cells in the row below the last nonzero part of $\mu$.

\begin{figure}
\begin{center}
\epsfig{file=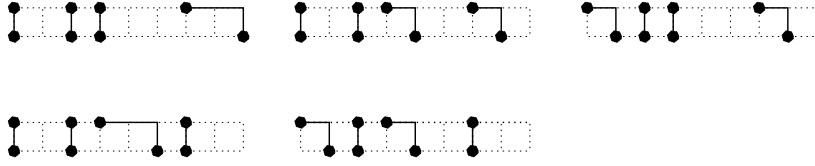,scale=0.6}
\caption{Computing $\M{h_2}(s_{(3110)})$ using abacus-histories.}
\label{fig:mult-h}
\end{center}
\end{figure}

Next consider how to compute $h_m^{\perp}(s_{\mu})$. Recall that the answer
is the sum of all $s_{\nu}$ where $\nu$ can be obtained by \emph{removing}
a horizontal strip of $m$ cells from the Ferrers diagram of $\mu$. To execute
this action on an abacus, we move various beads \emph{west} a total
of $m$ positions, avoiding collisions with the original locations of the
beads. Then we move every bead south one step to represent the passage 
of time.  For example, Figure~\ref{fig:h-perp} shows the abacus-histories 
encoding the computation
\[ h_2^{\perp}(s_{(332)})=s_{(330)}+s_{(321)}. \]
Note that $\mu$ need not be padded with zero parts when using this rule.

\begin{figure}
\begin{center}
\epsfig{file=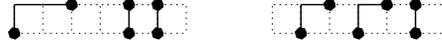,scale=0.6}
\caption{Computing $h_2^{\perp}(s_{(332)})$ using abacus-histories.}
\label{fig:h-perp}
\end{center}
\end{figure}

\subsection{Abacus Versions of $\M{e_m}$ and $e_m^{\perp}$}
\label{subsec:abacus-e}

Now we describe abacus implementations of $\M{e_m}$ and $e_m^{\perp}$.
Recall $\M{e_m}(s_{\mu})$ is the sum of all $s_{\lambda}$ where $\lambda$
is obtained by adding a vertical strip of $m$ cells to the Ferrers
diagram of $\mu$.
Comparing the frontiers of $\mu$ and $\lambda$, we see that adding such
a vertical strip corresponds to simultaneously moving $m$ distinct
beads east one step each. Beads cannot collide on the new abacus, but
a bead is allowed to move into a position vacated by another bead.

To record this move in an abacus-history diagram, start with an $N$-bead
abacus for $\mu$ where $\mu$ is padded with at least $m$ zero parts.
In all possible ways, pick a set of $m$ beads that each move southeast
one step, while the remaining beads move south one step with no collisions.
For example, Figure~\ref{fig:mult-e} shows the abacus-histories encoding
the computation
\[ \M{e_2}(s_{(41100)})=
 s_{(41111)}+s_{(42110)}+s_{(51110)}+s_{(42200)}+s_{(52100)}. \] 
Note that the $m$ zero parts are needed to accommodate a vertical strip
that we might add below the last nonzero part of $\mu$.

\begin{figure}
\begin{center}
\epsfig{file=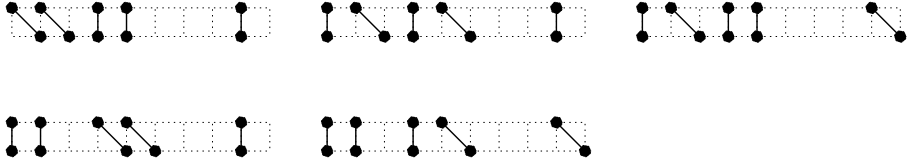,scale=0.6}
\caption{Computing $\M{e_2}(s_{(41100)})$ using abacus-histories.}
\label{fig:mult-e}
\end{center}
\end{figure}

Next, $e_m^{\perp}(s_{\mu})$ is the sum of all $s_{\nu}$ where $\nu$
is obtained from the Ferrers diagram of $\mu$ by removing a vertical strip of 
$m$ cells. Starting with the abacus for $\mu$, we pick a set of $m$ beads
that each move southwest one step, while the remaining beads move south
one step with no collisions. For example, Figure~\ref{fig:e-perp} shows
the abacus-histories encoding the computation
\[ e_2^{\perp}(s_{(442)})=s_{(431)}+s_{(332)}. \]
Note that $\mu$ need not be padded with zero parts when using this rule.

\begin{figure}
\begin{center}
\epsfig{file=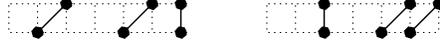,scale=0.6}
\caption{Computing $e_2^{\perp}(s_{(442)})$ using abacus-histories.}
\label{fig:e-perp}
\end{center}
\end{figure}

\subsection{Effects of $\omega$ and $\CO$}
\label{subsec:abacus-omega}

We know $\omega(s_{\mu})=s_{\mu'}$, where the Ferrers diagram of $\mu'$
is found by transposing the diagram of $\mu$. This transposition
interchanges the roles of north and east steps on the frontier
of $\mu$ and reverses the order of these steps. So, $\omega$ acts
on the doubly-infinite abacus for $\mu$ by interchanging
beads and gaps and reversing the abacus. Let us call this move
an \emph{abacus-flip}.

Now suppose we know a description of an operator $G$ in terms
of moves on an abacus. Then $\CO(G)=\omega\circ G\circ\omega$
acts on the abacus for $s_{\mu}$ by doing an abacus-flip,
then doing the moves for $G$, then doing another abacus-flip.
Therefore, we obtain a description of the operator $\CO(G)$ from the
given description of $G$ by interchanging the roles of beads
and gaps, and interchanging the roles of east and west.

For example, consider $G=\M{h_m}$ and $\CO(G)=\M{e_m}$.
We know $G$ acts on the abacus for $s_{\mu}$ by moving some \emph{beads}
$m$ steps \emph{east}, avoiding the original positions of all \emph{beads}.  
Therefore, we can say that $\CO(G)$ acts on the abacus for $s_{\mu}$
by moving some \emph{gaps} $m$ steps \emph{west} avoiding the
original positions of all \emph{gaps}. One may check that this description
of $\M{e_m}$ (involving gap motions) is equivalent to the description 
given earlier (involving bead motions). 
When composing several operators to build
bigger abacus-histories, it is often easier to work with descriptions that
always move beads rather than gaps.

\subsection{Abacus Version of $\Sop_m$}
\label{subsec:abacus-Sm}

We develop an initial abacus-history implementation of the operator
$\Sop_m$ based on formula~\eqref{eq:defS-alg}, which will subsequently
be simplified using a sign-reversing involution on abacus-histories.
Given any partition $\mu$, we know by~\eqref{eq:defS-alg} that
\[ \Sop_m(s_{\mu})=\sum_{c\geq 0} (-1)^c h_{m+c}e_c^{\perp}(s_{\mu}). \]
To compute this with an abacus-history, start with the abacus for $\mu$
(padding $\mu$ with a zero part if needed) and perform the following
steps in all possible ways. Choose an integer $c\geq 0$.
In the first time step, 
apply $e_c^{\perp}$ by moving $c$ distinct beads one step 
southwest while the remaining beads move one step south with no collisions.
In the second time step, apply $h_{m+c}$ by moving some beads a total of $m+c$ 
steps east, never moving into a position occupied by a bead at the start
of this time step; then move all beads one step south. Record
a term $(-1)^c$ times the Schur function corresponding to the final abacus.
For example, Figure~\ref{fig:bern1} shows how to compute
$\Sop_1(s_{(3110)})$ via abacus-histories. When $c=0$, we obtain the
three positive objects labeled A, B, C; when $c=1$, we obtain the
nine negative objects labeled D through L; and so on. Adding up the
23 signed Schur functions encoded by these abaci, there is massive
cancellation leading to the answer $-s_{(2211)}$.

\begin{figure}[h]
\begin{center}
\epsfig{file=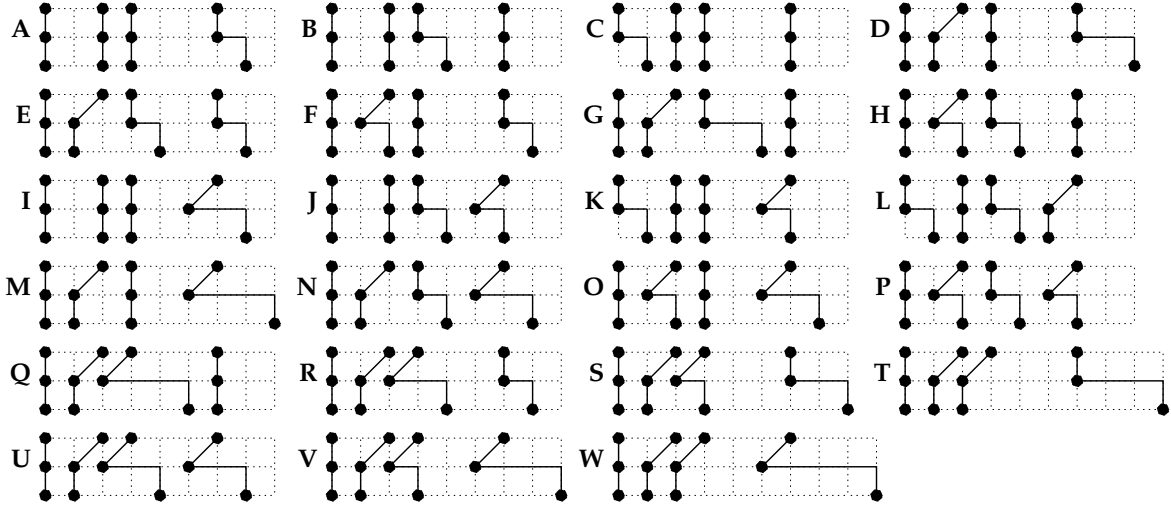,scale=0.6}
\caption{Initial computation of $\Sop_1(s_{(3110)})$ using abacus-histories.}
\label{fig:bern1}
\end{center}
\end{figure}

We now introduce a sign-reversing involution on abacus-histories that
explains the cancellation in the last example. Suppose an abacus-history
appearing in the computation of $\Sop_m(s_{\mu})$ contains a bead that
moves southwest, then moves east $i>0$ steps and then south, as shown
on the left in Figure~\ref{fig:bcancel}. By changing the first two
steps from southwest-east to a single south step, the bead now moves
% NL 7/2020: changed the description of $e$ (ref1,#1).
% Note the gap marked "e" might contain a bead in the first time interval,
% as shown by the Q <-> G pairing in Figure 6. So the description must
% be worded carefully.
as shown on the right in Figure~\ref{fig:bcancel}, where the $e$ denotes
a gap on the abacus that no bead visits during the second time interval.
Conversely, if a bead initially moves south and then has a gap to its left
that no bead visits, then we can
replace this initial south step with a southwest step followed by an east
step. These path modifications change $c$ by $1$ and hence change the sign 
of the abacus-history, while preserving the requirement of taking $m+c$
total east steps in the second time interval. The involution acts on
an abacus-history by scanning for the leftmost occurrence of one of
the patterns in Figure~\ref{fig:bcancel} and replacing it with the other
pattern. It is clear that doing the involution twice restores the 
original object. The fixed points of the involution (which could be
negative) consist of all abacus-histories avoiding both patterns
in Figure~\ref{fig:bcancel}.  
For example, applying the involution to the objects in Figure~\ref{fig:bern1}
produces the following matches:
\[ A\leftrightarrow F,\ 
B\leftrightarrow H,\ 
C\leftrightarrow K,\ 
D\leftrightarrow S,\ 
E\leftrightarrow R,\ 
G\leftrightarrow Q,\ 
I\leftrightarrow O,\ 
J\leftrightarrow P,\ 
M\leftrightarrow V,\ 
N\leftrightarrow U,\ 
T\leftrightarrow W. 
\]
We are left with the single negative fixed point $L$, confirming that 
$\Sop_1(s_{(3110)})=-s_{(2211)}$.

\begin{figure}
\begin{center}
\epsfig{file=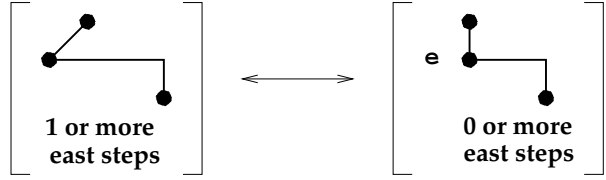,scale=0.8}
\caption{Cancellation move for abacus-histories appearing in $\Sop_m(s_{\mu})$.}
\label{fig:bcancel}
\end{center}
\end{figure}

We can now prove a formula for $\Sop_m(s_{\mu})$ that interprets
$\Sop_m$ as a \textbf{bead-creation operator} for abacus-histories.
As consequences of this formula, we can finally justify
equations~\eqref{eq:s-create},~\eqref{eq:S-commute}, and the technique
for computing $\Sop_m(s_{\mu})$ used in Example~\ref{ex:computeS}.  To
state the formula, we need some preliminary definitions.  For any
partition $\mu$, assign a \emph{label} and a \emph{sign} to each gap
on the abacus for $\mu$ as follows.  Given a gap with $g$ gaps
strictly to its left and $b$ beads to its right, let this gap have
label $g-b$ and sign $(-1)^b$. 
Here is another way to compute the gap
labels.  The gap to the right of the rightmost bead on the abacus for
$\mu$ has label $\mu_1$.  Any gap $i$ positions to the right
(resp. left) of this gap on the abacus has label $\mu_1+i$
(resp. $\mu_1-i$), as is readily checked.

\begin{theorem}\label{thm:create-bead}
For any partition $\mu$ and integer $m$, $\Sop_m(s_{\mu})$ is $0$
if no gap in the abacus for $\mu$ has label $m$.  
Otherwise, we compute $\Sop_m(s_{\mu})$ by filling the unique
gap labeled $m$ with a new bead, then multiplying the Schur function
for the new abacus by the sign of this gap.  
\end{theorem}
\begin{proof} 
First compute $\Sop_m(s_{\mu})$ by generating a collection of signed
abacus-histories, as described at the start of this section.
Next apply the sign-reversing involution to cancel out pairs of objects.
We must now analyze the structure of the fixed points
that remain. All fixed points avoid occurrences of both patterns
shown in Figure~\ref{fig:bcancel}. 
From our abacus-history characterization of the action of
$\Sop_m(s_\mu)$, there are only two other possible move patterns for a
bead starting in the abacus for $s_\mu$. The first is for a bead to
move south and then east zero or more east steps, but there must be
another bead southwest of this bead's starting point
(i.e., in the position marked $e$ on the right side of
Figure~\ref{fig:bcancel}). The second is for a bead to move
southwest and then south with no intervening east steps. In the
second case, we refer to this pair of steps as a \emph{zig-down move}
and say that the bead \emph{zigs down}.

Given a fixed point, suppose there is
a bead $Q$ on the input abacus that makes a zig-down move.
On one hand, suppose there is a bead $P$ immediately to the left of $Q$
(i.e., with no gaps in between).  Then $P$ must also zig down, since there is 
no room to do anything else. On the other hand, suppose
$R$ is the next bead somewhere to the right of $Q$ (if any).
There must be a vacancy immediately southwest of $R$'s initial position 
(since $Q$ zigs down). It follows that $R$ must also zig down to avoid the 
two forbidden patterns. 
Define a \emph{block} of beads on an abacus to be a maximal set of
beads with no gaps between any pair of them.
Iterating our two preceding observations, we conclude that \emph{if
  some bead zigs down in a fixed point, then all beads to its right
  and all beads to its left in its block must also zig down}.

Next consider a bead $S$ on the input abacus that has $g>0$ gaps
to its right followed by another bead $T$ that does not zig down.
On one hand, $S$ does not zig down (or else $T$ would too).
On the other hand, for $T$ to avoid the second pattern in 
Figure~\ref{fig:bcancel}, bead $S$ must move south, then $g$ steps east,
then south. We now see that all fixed points for $\Sop_m(s_{\mu})$ must
have the following structure. There are zero or more beads at the right
end that all zig down, starting with some bead $Q$ at the beginning
of a block of beads. The motions of all remaining beads
are uniquely determined (they move east as far as possible), 
except for the next bead $P$ to the left of $Q$.  
If beads $P$ and $Q$ are separated by $g>0$ gaps
on the input abacus, then $P$ has the option of moving south,
then $i$ steps east, then south, for any $i$ satisfying $0\leq i<g$.  
As a special case, if no beads zig down, then $P$ is the rightmost
bead on the abacus, which can move $i$ steps east for any $i\geq 0$.
For example, Figure~\ref{fig:fixed} illustrates some of the fixed points
for $\Sop_m(s_{(888442210)})$ for various choices of $m$ (compare
to Example~\ref{ex:computeS}).

\begin{figure}
\begin{center}
\epsfig{file=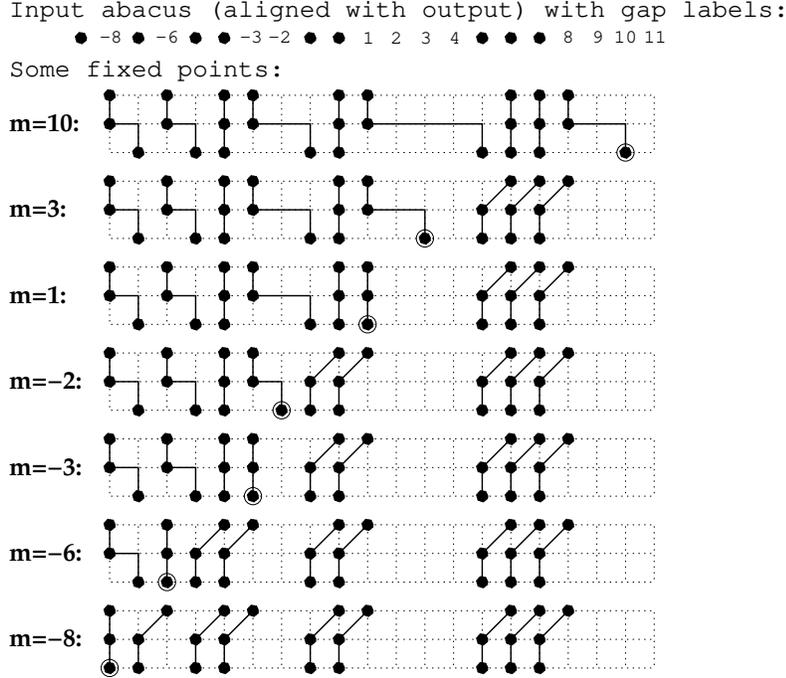,scale=0.6} 
\caption{Fixed points in the computation of $\Sop_m(s_{(888442210)})$
 for various values of $m$. Newly created beads are circled.}
\label{fig:fixed}
\end{center}
\end{figure}

It is visually evident from this example that, for general $\mu$ and
any fixed point of $\Sop_m(s_{\mu})$, the output abacus 
arises from the input abacus by shifting every bead one position
to the left (which corresponds to deleting a zero part from the
end of $\mu$) and then filling one gap with a new bead.
Suppose this gap has $g$ gaps
strictly to its left and $b$ beads to its right on the abacus for
$\mu$. The label of this gap is, by definition, $g-b$. To complete the
proof, we need only confirm that $g-b=m$. By our characterization of
fixed points, there are $b$ southwest steps in time interval $1$ (since
all beads to the right of this gap zig down). By~\eqref{eq:defS-alg}
and our observations at the beginning of this section, these southwest
steps arise from the action of $e_b^{\perp}$. As such, there must be
$m+b$ east steps arising from the subsequent action of $h_{m+b}$. But,
as illustrated in Figure~\ref{fig:fixed}, these east steps are in
bijection with the gaps to the left of the new bead. So the number of
gaps $g$ is $m+b$. It follows that $g-b=(m+b)-b=m$, as needed.
\end{proof}

Using the second method of computing gap labels,
we see that for any partition $\mu$
and any integer $m\geq\mu_1$, $\Sop_m(s_{\mu})=+s_{(m,\mu)}$.
Iterating this result starting with $s_0=1$, we obtain~\eqref{eq:s-create}.
Similarly, formula~\eqref{eq:S-commute} can be deduced quickly from
Theorem~\ref{thm:create-bead} by the following abacus-based proof.
The left side of~\eqref{eq:S-commute} acts on $s_{\mu}$
by first creating a new bead in the gap labeled $n$, then creating
a new bead in the gap now labeled $m$ (returning zero if either gap
does not exist). One readily checks that creating a new bead in
the gap labeled $n$ has the effect of decrementing all remaining gap
labels. Thus, $\Sop_m\circ\Sop_n$ acts by filling the gap labeled $n$,
then filling the gap \emph{originally} labeled $m+1$, if these gaps exist.
Similarly, $\Sop_{n-1}\circ\Sop_{m+1}$ acts by filling the gap labeled $m+1$,
then filling the gap \emph{originally} labeled $n$, if these gaps exist.
These actions are the same except for the order in which the two new
beads are created, which causes the two answers to differ by a sign change.
As a special case, when $n=m+1$, both sides of~\eqref{eq:S-commute}
output zero because the second operator on each side tries to fill a gap that
no longer exists. This completes the proof of~\eqref{eq:S-commute}.  
Finally, the computations in Example~\ref{ex:computeS} 
are now justified by combining formulas~\eqref{eq:s-create},~\eqref{eq:Salpha},
and~\eqref{eq:S-commute}.

By applying the results in Section~\ref{subsec:abacus-omega}, we 
also obtain a dual theorem characterizing the action of $\CO(\Sop_m)$ as a 
``gap-creation operator'' or a ``bead-destruction operator.''
Specifically, given a bead with $b$ beads
strictly to its right and $g$ gaps to its left, let this bead have label $b-g$ and sign $(-1)^g$.
(Note that this labeling of \emph{beads} is related to, but different from, 
our prior labeling scheme for \emph{gaps}.)
If the abacus for $\mu$ has a bead labeled $m$, then
$\CO(\Sop_m)(s_{\mu})$ is the sign of this bead times $s_{\nu}$,
where we get the abacus for $\nu$ by replacing the bead labeled $m$ by a gap.
If the abacus for $\mu$ has no bead labeled $m$, then
$\CO(\Sop_m)(s_{\mu})$ is zero.

\subsection{Abacus Versions of $\Hop_m$ and $\Cop_m$}
\label{subsec:abacusHC}

With Theorem~\ref{thm:create-bead} in hand, we can describe how
to compute $\Hop_m(s_{\mu})$ and $\Cop_m(s_{\mu})$ using abacus-histories.
Our implementation of $\Hop_m$ is based on formula~\eqref{eq:defH-alg}.
Starting with the abacus for $\mu$, do the following steps in all possible
ways. Choose an integer $c\geq 0$. In time step $1$, apply $h_c^{\perp}$
by moving some beads a total of $c$ steps west (avoiding original bead
positions), then moving all beads one step south. In time step $2$,
apply $\Sop_{m+c}$ by creating a new bead in the gap now labeled $m+c$ (if any).
The Schur function corresponding to the new abacus is weighted
by $q^c$ times the sign of the gap where the new bead was created.

\begin{figure}[ht]
\begin{center}
\epsfig{file=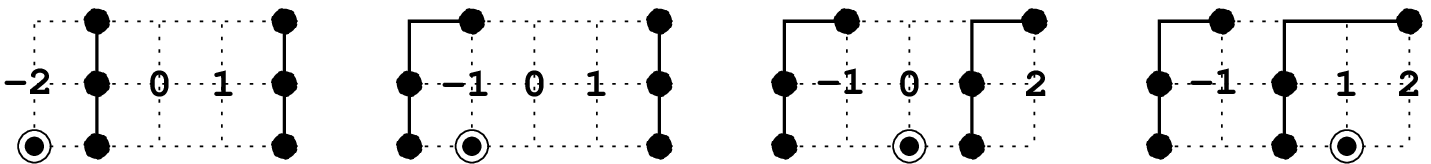,scale=0.8} 
\caption{Computing $\Hop_{-2}(s_{(31)})$ using abacus-histories.}
\label{fig:H-ex}
\end{center}
\end{figure}

For example, by drawing the objects in Figure~\ref{fig:H-ex}, we find
\begin{equation}\label{eq:H-ex1}
 \Hop_{-2}(s_{(31)})=+s_{(200)}-q^1s_{(200)}-q^2s_{(110)}+q^3s_{(110)}. 
\end{equation}
In these abacus-histories, we have included gap labels in the middle row
and circled the new beads created by $\Sop_{m+c}$. 
By making similar pictures, one can check that 
% done in WC18, pg. 95.
\begin{equation}\label{eq:H-ex2}
 \Hop_1(s_{(31)})=-s_{(221)}+q^1s_{(221)}+q^2s_{(320)}
   +q^2s_{(311)}+q^3s_{(410)}. 
\end{equation}
Note that these answers are neither Schur-positive nor Schur-negative.
We observe that every original bead takes two consecutive south steps
in these abacus-histories. By combining these steps into a single south step,
we can shorten the time needed for the $\Hop_m$ operator from two time
steps to one time step. We do this from now on.

We can compute $\Cop_m(s_{\mu})$ by generating exactly the same collection of
abacus-histories used for $\Hop_m(s_{\mu})$. The only difference is
that each weight $q^c$ is replaced by $q^{-c}$, and the final answer
is multiplied by the global factor $(-1/q)^{m-1}$. For example,
\begin{align*}
  \Cop_{-2}(s_{(31)}) &= -q^3s_{(200)}+q^2s_{(200)}+q^1s_{(110)}-q^0s_{(110)};
 \\ \Cop_1(s_{(31)})  &=-s_{(221)}+q^{-1}s_{(221)}+q^{-2}s_{(320)}
    +q^{-2}s_{(311)}+q^{-3}s_{(410)}. 
\end{align*}

\subsection{Abacus Version of $\Bop_m$}
\label{subsec:abacusB}

Finally, we describe how to compute $\Bop_m(s_{\mu})$ using abacus-histories.
Since $\Bop_m=\CO(\Hop_m)=\omega\circ\Hop_m\circ\, \omega$, one approach is to
calculate $\Hop_m(s_{\mu'})$ as described earlier, then conjugate all
partitions in the resulting Schur expansion. 
For example, using~\eqref{eq:H-ex1} and~\eqref{eq:H-ex2}, we compute:
\begin{align*}
 \Bop_{-2}(s_{(211)}) &=
 +s_{(11)}-q^1s_{(11)}-q^2s_{(2)}+q^3s_{(2)};  
\\ \Bop_{1}(s_{(211)}) &=
 -s_{(32)}+q^1s_{(32)}+q^2s_{(221)} +q^2s_{(311)}+q^3s_{(2111)}. 
\end{align*}

Alternatively, we can use the formula
\[ \Bop_m=\sum_{d\geq 0} q^d\CO(\Sop_{m+d})\circ e_d^{\perp} \]
proved in Section~\ref{subsec:alg-conj-omega}. Starting with the 
doubly-infinite abacus for $\mu$, do the following steps in all possible ways.
Choose an integer $d\geq 0$. In time step $1$, apply $e_d^{\perp}$
by moving $d$ distinct beads one step southwest while the remaining beads
move one step south with no collisions. In time step $2$,
apply $\CO(\Sop_{m+d})$ by destroying the bead with label $m+d$ (if any).
The Schur function corresponding to the new abacus is weighted by $q^d$
times the sign of the destructed bead.

\section{Abacus-History Models for $H_{\alpha}$, 
 $C_{\alpha}$, $B_{\alpha}$, etc.}
\label{sec:BCH}

\subsection{Combinatorial Formulas}
\label{subsec:BCH-formulas}

Now that we have abacus-history interpretations for the 
operators $\Sop_m$, $\Hop_m$, $\Cop_m$, $\Bop_m$, $h_m^{\perp}$, etc.,
we can build abacus-history models giving the Schur expansion 
when any finite sequence of these operators is applied to any Schur function.
We simply concatenate the diagrams for the individual operators
in all possible ways and sum the signed, weighted Schur functions
corresponding to the final abaci. We illustrate this process here by
describing combinatorial formulas for $H_{\alpha}$, $C_{\alpha}$, 
$B_{\alpha}$, and the analogous symmetric functions 
$\Hop_{\alpha}(s_{\mu})$, $\Cop_{\alpha}(s_{\mu})$, and
$\Bop_{\alpha}(s_{\mu})$ obtained by replacing $1$ by $s_{\mu}$ 
in~\eqref{eq:Halpha}, \eqref{eq:Calpha}, and~\eqref{eq:Balpha}.

Fix a sequence of integers $\alpha=(\alpha_1,\alpha_2,\ldots,\alpha_L)$.
We compute $H_{\alpha}=\Hop_{\alpha_1}\circ\cdots\circ \Hop_{\alpha_L}(1)$
using abacus-histories that take $L$ time steps. We start with an empty
abacus (corresponding to the input $s_{(0)}=1$) where the gap in each
position $i\geq 0$ has label $i$.  In time step 1, we cannot
move any beads west, so we create a new bead in the gap labeled $\alpha_L$.
In time step 2, we choose $c_2\geq 0$, move the lone bead west $c_2$ steps 
and south once, then create a new bead in the gap now labeled 
$\alpha_{L-1}+c_2$.  In time step 3, we choose $c_3\geq 0$, move the two
beads west a total of $c_3$ steps and south, then create a new bead
in the gap now labeled $\alpha_{L-2}+c_3$. And so on. If at any stage
there is no gap with the required label, then that particular diagram
disappears and contributes zero to the answer. If the diagram survives
through all $L$ time steps, then its final abacus contributes a Schur
function weighted by $q^{c_2+c_3+\cdots+c_L}$ times the signs arising
from all the bead creation steps. Thus, the final power of $q$ is the
total number of west steps taken by all the beads, while the final
sign is $-1$ raised to the total number of beads to the right of newly
created beads in all time steps. The computation for $\Hop_{\alpha}(s_{\mu})$
is the same, except now we start with the abacus for $\mu$ instead of
an empty abacus. Here we might move some beads $c_1\geq 0$ steps west
in the first time interval, leading to the creation of a new bead
in the gap now labeled $\alpha_L+c_1$. When $\mu=0$ we must have $c_1=0$.

The following remarks can aid in generating the diagrams for $H_{\alpha}$.
The \emph{default starting positions} for the new beads are
$\alpha_L,\alpha_{L-1}+1,\alpha_{L-2}+2,\ldots,\alpha_1+L-1$. 
These are the positions (not gap labels) on the initial abacus 
where new beads would appear if all $c_i$ were zero. 
(This follows since gap labels coincide with position numbers on
an empty abacus, and each bead creation decrements all current gap labels.) 
The \emph{actual starting positions} for the new beads are
\[ \alpha_L+c_1,\alpha_{L-1}+1+c_2,\alpha_{L-2}+2+c_3,\ldots,\alpha_1+L+c_L; \]
these are obtained by moving each default starting position east 
by the number of west steps in the preceding row. 
% CHECKED in an example, WC18 pg. 96. Positions versus gap labels are ok
%  because the label of a gap in a particular position always decrements
%  compared to the label for that position in the preceding time step,
%  no matter how beads are moving on the abacus.

\begin{example}
Figure~\ref{fig:H123} uses abacus-histories to compute
$H_{(123)}=\Hop_1(\Hop_2(\Hop_3(1)))$.  For brevity, we omit the top
rows, which have no beads in any positive position.  
All three new beads have
default starting position $3$.  There are 18 objects in all, but we
find two pairs of objects that cancel (C with F, and I with N).  
% confirmed.  
We are left with
\begin{align*}
 H_{(123)}
  = &q^8s_{(600)}+(q^6+q^7)s_{(510)}+(q^6+q^5+q^4-q^3)s_{(420)}
 \\ &+q^5s_{(411)}+q^5s_{(330)}+(q^4+q^3-q^2)s_{(321)}+(q^2-q)s_{(222)}.  
\end{align*} 
\end{example}

\begin{figure}
\begin{center}
\epsfig{file=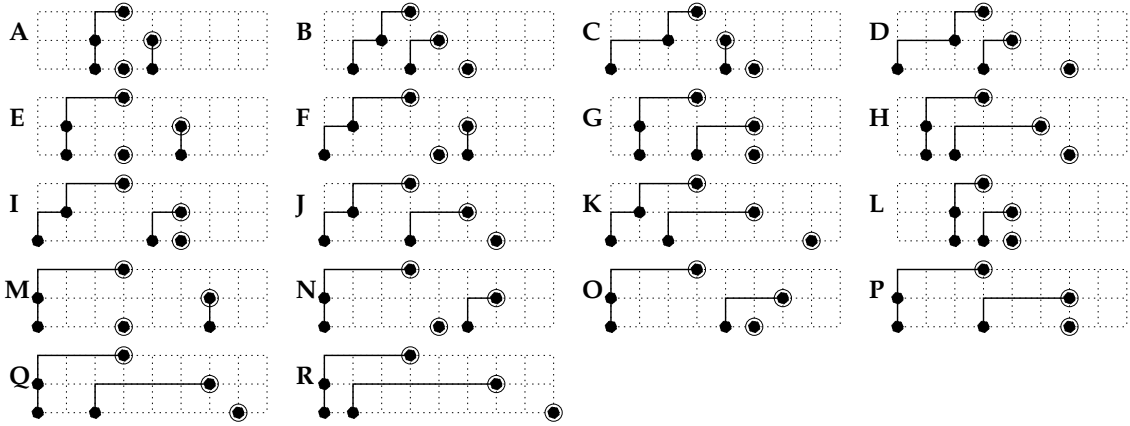,scale=0.6} 
\caption{Computing $H_{(123)}$ using abacus-histories.}
\label{fig:H123}
\end{center}
\end{figure}

We compute $C_{\alpha}$ (resp. $\Cop_{\alpha}(s_{\mu})$) by making
exactly the same abacus-histories used for $H_{\alpha}$
(resp. $\Hop_{\alpha}(s_{\mu})$). The only difference is that
the $q$-weight of each abacus-history is now $q^{-(c_1+\cdots+c_L)}$ 
and the final answer must be multiplied by $(-1/q)^{|\alpha|-\ell(\alpha)}$,
where $|\alpha|=\alpha_1+\cdots+\alpha_L$ and $\ell(\alpha)=L$.

Since $\Bop_m=\CO(\Hop_m)$ for every integer $m$, we can compute
$\Bop_{\alpha_1}\circ\cdots\circ\Bop_{\alpha_L}(s_{\mu})$
by applying $\Hop_{\alpha}$ to $s_{\mu'}$ as described above, then
conjugating all partitions in the resulting Schur expansion.
Alternatively, we can chain together the moves for the $\Bop_{\alpha_i}$
described at the end of Section~\ref{subsec:abacusB}.
Beware that $B_{\alpha}$ (as defined in~\cite{HMZ12}) is found
by starting with $1$ and applying $\Bop_{\alpha_1}$, 
$\Bop_{\alpha_2}$, $\ldots$, $\Bop_{\alpha_L}$ in this order.

% NL 7/2020: added this remark for referee 1, #2 and #3.
\begin{remark}
Python and SageMath code for computing with creation operators
and abacus-histories is available on the second author's 
website~\cite{GSW-code}. For $k=2,3,4,5,6$, the computation
of $H_{\alpha}$ for $\alpha=(3^k)$ involves $4$, $27$, $338$,
$6262$, and $168312$ abacus-histories, respectively. 
The time required (in seconds) was $0.04$, $0.05$, $0.08$,
$0.68$, and $24.61$, respectively. As another example,
$C_{(5,1,4,2,3,1)}$ has $16682$ terms and takes 5 seconds to compute.
\end{remark}

\subsection{Application to Three-Row Hall--Littlewood Polynomials}
\label{subsec:HLpoly}

% NL 7/2020: inserted allusion to Kostka--Foulkes polynomials (ref2#11).
% Question: Is our $H_{\nu}$ the Hall--Littlewood $P_{\nu}$ or $Q_{\nu}$?
%  Is it correct to use charge here (vs. cocharge)?
As we have seen, for general $\alpha$ the Schur expansion of $H_{\alpha}$
has a mixture of positive and negative terms. But for partitions
$\nu$, a celebrated theorem of Lascoux and Sch\"utzenberger~\cite{LS-HL}
shows that all Schur coefficients of $H_{\nu}$ are polynomials in $q$ with 
nonnegative integer coefficients (the \emph{Kostka--Foulkes polynomials}).
According to this theorem, the 
coefficient of $s_{\lambda}$ in $H_{\nu}$ is the sum of $q^{\charge(T)}$
over all semistandard Young tableaux $T$ of shape $\lambda$ and content
$\nu$, where charge is computed from $T$ by an explicit combinatorial
rule (see~\cite{Butler},~\cite[pg. 242]{Mac},
and~\cite[Sec. 1.7]{manivel} for more details).
% NL 7/2020: added reference to rigged configurations (ref2#12).
Kirillov and Reshetikhin~\cite{KR1,KR2} gave another combinatorial formula for
the Kostka--Foulkes polynomials as a sum over \emph{rigged configurations}
weighted by a suitable charge statistic. A detailed survey of combinatorial
formulas for Hall--Littlewood polynomials and their applications to 
representation theory may be found in~\cite{DLT}.
% NL 7/2020: added DLT reference above (ref2#13).

Our combinatorial formula for $H_{\alpha}$ based on abacus-histories 
holds for general integer sequences $\alpha$, but it is not manifestly
Schur-positive when $\alpha$ happens to be an integer partition. On the
other hand, our $q$-statistic (the total number of west steps in the
object) is much simpler to work with compared to the complicated charge
statistic on tableaux. As an application of our abacus-history model,
we now give a simple bijective proof of the Schur-positivity of
$H_{\nu}$ when $\nu$ is a partition with at most three parts.

If $\nu$ has only one part, then it is immediate that
$H_{\nu}=\Hop_{\nu_1}(1)=s_{\nu_1}$. Next suppose $\nu=(\nu_1\geq \nu_2)$ 
has two parts. When computing $H_{\nu}$ via abacus-histories,
the default starting positions of the two beads are $\nu_2$ and 
$\nu_1+1>\nu_2$. When the second bead is created, the first bead
has moved to a column $\nu_2-c_2$ for some $c_2\geq 0$, and the
second bead actually starts in column $\nu_1+1+c_2>\nu_2-c_2$.
This means that all abacus-histories for $H_{\nu}$ have positive sign, 
and our formula is already Schur-positive in this case.

Now let $\nu=(\nu_1\geq\nu_2\geq\nu_3)$ have three parts. As before,
since $\nu_2\geq\nu_3$, the second bead always gets created to the
right of the first bead's current column. For similar reasons, the
bead created third must start to the right of the first bead in the lowest row.
But it is possible that this third bead appears to the left of the
second bead's position in that row, leading to a negative object 
that must be canceled.

We cancel these objects using the involution suggested in 
Figure~\ref{fig:Hcancel}. Given a negative object as just described,
let $k$ be the distance between the new bead in the lowest row and
the bead to its right. Let $a,b,c\geq 0$ count west steps
as shown on the left in the figure, so bead $1$'s path
is $\W^a\Sst\W^b\Sst$ and bead $2$'s path is $\W^c\Sst$. The involution acts
by replacing $a$ by $a-k$ and $b$ by $b+k$, which causes bead $2$'s
actual starting position to move left $k$ columns and bead $3$'s actual
starting position to move right $k$ columns. This action matches the given
negative object with a positive object having the same number of west steps 
(hence the same $q$-power) and the same bead positions on the final abacus 
(hence the same Schur function).

\begin{figure}
\begin{center}
\epsfig{file=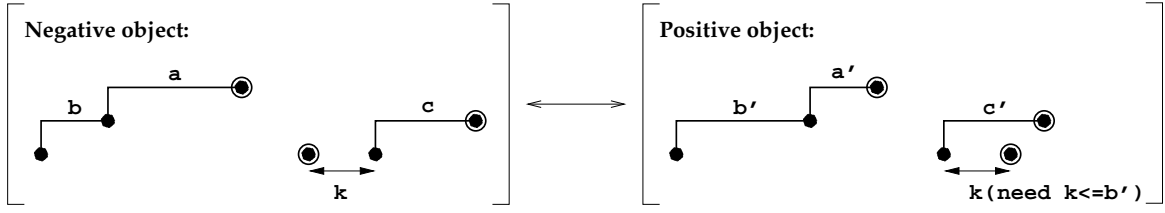,scale=0.7}
\caption{An involution on abacus-histories with three beads.}
\label{fig:Hcancel}
\end{center}
\end{figure}

Going the other way, consider a positive object (as shown on the
right in Figure~\ref{fig:Hcancel}) where the new bead in the bottom
row is $k$ columns to the right of the second bead, bead $1$'s path is 
$\W^{a'}\Sst\W^{b'}\Sst$, and bead $2$'s path is $\W^{c'}\Sst$. If $k\leq b'$,
then the involution acts by replacing $a'$ by $a'+k$ and $b'$ by $b'-k$,
causing the other two paths to switch places as before. If $k>b'$,
then this positive object is a fixed point of the involution.

To see that this involution really works, we must check a few items.
Fix an arbitrary negative object with notation as in Figure~\ref{fig:Hcancel}.
First, we must show $k\leq a$, since $a'$ is not allowed to be negative.
On one hand, the actual starting position of the new bead in the bottom row
is $\nu_1+2+b+c$. On the other hand, the bead created in the middle row
starts in position $\nu_2+1+a$ and ends in the bottom row
in position $\nu_2+1+a-c$. Therefore,
\begin{equation}\label{eq:k-formula} % WC18, pg. 4
 k=(\nu_2+1+a-c)-(\nu_1+2+b+c)=a-(b+2c+1+\nu_1-\nu_2). 
\end{equation}
Since $b,c\geq 0$ and $\nu_1\geq\nu_2$, the quantity subtracted from $a$
is strictly positive, so we actually have $k<a$. 

Second, we show that applying the involution to a negative object
does not cause the first two beads to collide in the middle row.  
After replacing $a$ by $a'=a-k$, the first bead moves from the top row
to the middle row in position $\nu_3-a'=\nu_3-a+k$. The bead created
in the middle row now starts in position $\nu_2+1+a'=\nu_2+1+a-k$
and moves left $c'=c$ steps to position $\nu_2+1+a-k-c$ before moving
down to the bottom row. Therefore, to avoid a bead collision, we need
$\nu_3-a+k<\nu_2+1+a-k-c$, or equivalently $2(a-k)>c+\nu_3-\nu_2-1$.  
Using~\eqref{eq:k-formula} to substitute for $a-k$, we need
$2(b+2c+1+\nu_1-\nu_2)>c+\nu_3-\nu_2-1$, which rearranges to
$2b+3c+3+2\nu_1-\nu_2-\nu_3>0$. %matches red edits in WC18,p.4.
This is true, since $b,c\geq 0$ and $\nu_1\geq\nu_2\geq\nu_3$.

Third, we show that applying the involution twice restores the original object.
This follows since the value of $k$ is the same for the object and its image,  
and $b'=b+k$ automatically satisfies $k\leq b'$. We have now proved that our
involution is well-defined and cancels all negative objects.

By analyzing the fixed points of this involution more closely, we can
prove the following explicit formula for the Schur coefficients of
$H_{\nu}$ when $\nu$ is a three-part partition.

\begin{theorem}
For all partitions $\nu=(\nu_1\geq\nu_2\geq \nu_3>0)$
and $\lambda=(\lambda_1\geq\lambda_2\geq\lambda_3\geq 0)$ 
such that $|\lambda|=|\nu|$, the coefficient of $s_{\lambda}$ in $H_{\nu}$ is 
\begin{equation}\label{eq:schur-coef}
 \sum_{b=0}^{\min(\lambda_1-\lambda_2,\lambda_2-\lambda_3,\nu_3-\lambda_3,
   \lambda_1-\nu_1)} q^{\nu_3-\lambda_3+\lambda_1-\nu_1-b}. 
\end{equation}
\end{theorem}
\begin{proof}
We fix $\lambda,\nu$ as in the theorem statement and enumerate the
positive fixed points of the involution. To obtain an uncanceled term 
$s_{\lambda}$ in the computation of $H_{\nu}$, the beads in the abacus-history 
must move as follows. The first bead starts in
position $\nu_3$ and ends in position $\lambda_3$ after moving along
some path $\W^a\Sst\W^b\Sst$. 
(We switch here to unprimed parameters for the positive fixed point.) 
The second bead starts in position
$\nu_2+1+a$ and ends in position $\lambda_2+1$ (since all negative objects
cancel) after moving along some path $\W^c\Sst$. 
The third bead starts in position $\nu_1+2+b+c$
and ends (without moving) in position $\lambda_1+2$.
We deduce that $\lambda_3=\nu_3-a-b$, $\lambda_2+1=\nu_2+1+a-c$,
and $\lambda_1+2=\nu_1+2+b+c$.  Since $\lambda$ and $\nu$ are fixed, 
the entire object is uniquely determined once we select the value of $b$. 
Eliminating $a$ and $c$, we see that the $q$-weight of the object is
\[ q^{a+b+c}=q^{\nu_3-\lambda_3+\lambda_1-\nu_1-b}. \] 

Which choices of $b$ are allowed? We certainly need $b\geq 0$,
as well as $a\geq 0$ and $c\geq 0$. Using $a=\nu_3-\lambda_3-b$
and $c=\lambda_1-\nu_1-b$, the conditions on $a$ and $c$ are
equivalent to $b\leq\nu_3-\lambda_3$ and $b\leq\lambda_1-\nu_1$.
We also need the first two beads not to collide in the middle row,
so we need $\nu_3-a<\nu_2+1+a-c$. This condition is equivalent to
$b+\lambda_3<\lambda_2+1$ and, hence, to $b\leq \lambda_2-\lambda_3$. 
Finally, letting $k=(\lambda_1+2)-(\lambda_2+1)$ be the distance
between the two rightmost beads in the bottom row, we need $k>b$
for this positive object to be a fixed point of the involution.
This condition rearranges to $b\leq\lambda_1-\lambda_2$. Combining
the five conditions on $b$ leads to the summation in the theorem statement.
\end{proof}

It is also possible to derive~\eqref{eq:schur-coef} starting from
the charge formula for the Schur expansion of $H_{\nu}$. But such a 
proof is quite tedious, requiring a messy case analysis due to the
complicated definition of the charge statistic.

One might ask if the involution for three-row shapes extends
to partitions $\nu$ with more than three parts. While the same involution
certainly applies to the first three rows of larger abacus-histories,
more moves are needed to eliminate all negative objects. Even in the
case of four-row shapes, the new cancellation moves are much more intricate
than the move described here. We leave it as an open question to reprove
the Schur-positivity of $H_{\nu}$ for all partitions $\nu$ via an explicit
involution on abacus-histories. It would also be interesting to find a
specific weight-preserving bijection between the set of fixed points for 
such an involution and the set of semistandard tableaux.

\section{Appendix: Proofs of Two Plethystic Formulas}
\label{sec:pleth-prf} 

This appendix proves the equivalence of the plethystic definitions
and the algebraic definitions of $\Sop_m$ and $\Cop_m$. We require
just three basic plethystic identities (see~\cite{expose} 
for a detailed exposition of plethystic notation including proofs
of these facts).  First, for any alphabets $A$ and $B$ and
any partition $\mu$, we have the plethystic addition rule
\[ s_{\mu}[A+B]=\sum_{\nu:\,\nu\subseteq\mu}
   s_{\nu}[A]s_{\mu/\nu}[B]. \]
Second, for formal variables $q$ and $z$ and partitions $\nu\subseteq\mu$,
% NAL 7/2020: I did switch to cases environment (recommended by Greg, JCTA).
\[ s_{\mu/\nu}[q/z]= \begin{cases}
 (q/z)^{|\mu/\nu|} & \mbox{ if $\mu/\nu$ is a horizontal strip; }
\\ 0 & \mbox{ otherwise.}\end{cases} \]
%\[ s_{\mu/\nu}[q/z]=
%  \left\{\begin{array}{ll}
% (q/z)^{|\mu/\nu|} & \mbox{ if $\mu/\nu$ is a horizontal strip; }
%\\ 0 & \mbox{ otherwise.}\end{array}\right. \]
Third, for all $\nu\subseteq\mu$,
\[ s_{\mu/\nu}[-1/z]
  =(-1)^{|\mu/\nu|}s_{\mu'/\nu'}[1/z]
  =\begin{cases}
 (-1/z)^{|\mu/\nu|} & \mbox{ if $\mu/\nu$ is a vertical strip; }
\\ 0 & \mbox{ otherwise.}\end{cases} \]
%  =\left\{\begin{array}{ll}
% (-1/z)^{|\mu/\nu|} & \mbox{ if $\mu/\nu$ is a vertical strip; }
%\\ 0 & \mbox{ otherwise.}\end{array}\right. \]
Recall that $\HS(c)$ (resp. $\VS(c)$) is the set of horizontal
(resp. vertical) strips with $c$ cells.

We prove the equivalence of the definitions~\eqref{eq:defS-pleth}
and~\eqref{eq:defS-alg} for $\Sop_m$ by showing that the two formulas
have the same action on the Schur basis. Taking $P=s_{\mu}$ 
in~\eqref{eq:defS-pleth} and using the rules above, we compute:
\begin{align*}
 \Sop_m(s_{\mu}) 
&= \left.s_{\mu}[X-(1/z)]\sum_{k\geq 0} h_kz^k\right|_{z^m}
=\left.\sum_{\nu\subseteq\mu}
        s_{\nu}[X]s_{\mu/\nu}[-1/z]\sum_{k\geq 0} h_kz^k\right|_{z^m}
\\&=\left.\sum_{c\geq 0}\sum_{\substack{\nu\subseteq\mu:\\\mu/\nu\in\VS(c)}}
   (-1/z)^{c}\sum_{k\geq 0} s_{\nu}h_kz^k\right|_{z^m}
=\sum_{c\geq 0}\sum_{\substack{\nu\subseteq\mu:\\\mu/\nu\in\VS(c)}}
   (-1)^{c} s_{\nu}h_{m+c}
\\&=\sum_{c\geq 0} (-1)^c \M{h_{m+c}}(e_c^{\perp}(s_{\mu})).  
\end{align*}
This agrees with~\eqref{eq:defS-alg}.

Now we prove the equivalence of~\eqref{eq:defH-pleth}
and~\eqref{eq:defH-alg}. Applying~\eqref{eq:defH-pleth} to $P=s_{\mu}$ gives:
\begin{align*}
 \Hop_m(s_{\mu}) 
&= \left.s_{\mu}[(X-1/z)+q/z]\sum_{k\geq 0} h_kz^k\right|_{z^m}
=\left.\sum_{\nu\subseteq\mu}
        s_{\nu}[X-1/z]s_{\mu/\nu}[q/z]\sum_{k\geq 0} h_kz^k\right|_{z^m}
\\&=\left.\sum_{c\geq 0}\sum_{\substack{\nu\subseteq\mu:\\\mu/\nu\in\HS(c)}} 
 (q/z)^c s_{\nu}[X-1/z]\sum_{k\geq 0} h_kz^k\right|_{z^m}
\\&=\left.\sum_{c\geq 0}q^c\sum_{\substack{\nu\subseteq\mu:\\\mu/\nu\in\HS(c)}} 
  s_{\nu}[X-1/z]\sum_{k\geq 0} h_kz^k\right|_{z^{m+c}}
=\sum_{c\geq 0}q^c\sum_{\substack{\nu\subseteq\mu:\\\mu/\nu\in\HS(c)}} 
   \Sop_{m+c}(s_{\nu})
\\&=\sum_{c\geq 0} q^c \Sop_{m+c}\left(
\sum_{\nu\subseteq\mu:\ \mu/\nu\in\HS(c)} s_{\nu}\right)
= \sum_{c\geq 0} q^c \Sop_{m+c}(h_c^{\perp}(s_{\mu})).
\end{align*}
This agrees with~\eqref{eq:defH-alg}.


\begin{thebibliography}{99} 
\bibitem{slinky}
 Per Alexandersson, \emph{The symmetric functions catalog}, online at 
\verb-https://www.math.upenn.edu/~peal/polynomials/gessel.htm#gesselSlinkyRule-
 (2020).

\bibitem{Butler}
 L. Butler, \emph{Subgroup lattices and symmetric functions},
  Mem. Amer. Math. Soc. \textbf{112} (1994).

\bibitem{DLT}
 J. D\'esarm\'enien, B. Leclerc, and J.-Y. Thibon,
 ``Hall--Littlewood functions and Kostka--Foulkes polynomials
 in representation theory,'' \emph{S\'em. Lotharingien de
 Combinatoire}, paper B32c (1994), 38 pages (electronic).

% not cited; but this reference has nice summary of KR results
%  on rigged configurations.
% \bibitem{fishel}
% Susanna Fishel, ``Statistics for special $q,t$-Kostka polynomials,''
%  \emph{Proc. Amer. Math. Soc.} \textbf{123} no. 10 (1995), 2961--2969.

\bibitem{GXZ11}
 Adriano Garsia, Guoce Xin, and Michael Zabrocki,
 ``Hall--Littlewood operators in the theory of parking functions 
   and diagonal harmonics,'' \emph{Intl. Math. Research Notices} (2011),
   article ID rnr060, 36 pages.

\bibitem{HMZ12} 
  James Haglund, Jennifer Morse, and Michael Zabrocki, ``A compositional
  shuffle conjecture specifying touch points of the Dyck path,''
  \emph{Canad. J. Math.}, \textbf{64} no. 4 (2012), 822--844.

\bibitem{JK}
 Gordon James and Adalbert Kerber, \emph{The Representation Theory of the 
 Symmetric Group},  Addison-Wesley (1981).

\bibitem{Jing91}
 Naihuan Jing, ``Vertex operators and Hall--Littlewood symmetric functions,''
 \emph{Adv. Math.} \textbf{87} (1991), 226--248.

\bibitem{Jing-thesis}
 Naihuan Jing, \emph{Vertex operators, symmetric functions and their
  $q$-deformations}, Yale University dissertation (1989).

\bibitem{KR1} A. N. Kirillov and N. Yu. Reshetikhin,
 \emph{Combinatorics, Bethe ansatz, and representations of the
 symmetric group}, Zap. Nauchn. Sem. Leningrad. Otdel. Mat. Inst. Steklov.
 (LOMI) \textbf{155} (1986), 50--64; English translation in
 \emph{J. Soviet Math.} \textbf{41} (1988), 916--924.

\bibitem{KR2} A. N. Kirillov and N. Yu. Reshetikhin, 
 \emph{The Bethe ansatz and the combinatorics of Young tableaux},
 Zap. Nauchn. Sem. Leningrad. Otdel. Mat. Inst. Steklov.
 (LOMI) \textbf{155} (1986), 65--115; English translation in
 \emph{J. Soviet Math.} \textbf{41} (1988), 925--955.

\bibitem{LS-HL}
 A. Lascoux and M.-P. Sch\"utzenberger, ``Sur une conjecture de 
  H. O. Foulkes,'' \emph{C. R. Acad. Sci. Paris S\'er. A--B} \textbf{286A}
  (1978),  323--324.

\bibitem{littlewood}
 D. E. Littlewood, \emph{The Theory of Group Characters} (second edition),
 Oxford University Press (1950).

\bibitem{loehr-abacus}
 Nicholas Loehr, ``Abacus proofs of Schur function identities,''
 \emph{SIAM J. Discrete Math.} \textbf{24} (2010), 1356--1370.

\bibitem{loehr-comb}
 Nicholas Loehr, \emph{Combinatorics} (second edition), CRC Press (2017).
 
\bibitem{expose}
 Nicholas Loehr and Jeffrey Remmel, ``A computational and combinatorial
  expos\'e of plethystic calculus,'' \emph{J Algebraic Combin.}
 \textbf{33} (2011), 163--198.

\bibitem{Mac} Ian Macdonald, \emph{Symmetric Functions and Hall Polynomials}
 (second ed.), Oxford University Press (1995).

\bibitem{manivel} Laurent Manivel, \emph{Symmetric functions, Schubert
 polynomials, and degeneracy loci}, American Mathematical Society (2001).

\bibitem{RemYoo} Jeffrey Remmel and Meesue Yoo,
 ``The combinatorics of the HMZ operators applied to Schur functions,''
 \emph{J. Comb.} \textbf{3} no. 3 (2012), 401--450.

\bibitem{sagan} Bruce Sagan, \emph{The Symmetric Group: Representations,
 Combinatorial Algorithms, and Symmetric Functions} (second edition),
 Springer Graduate Texts in Mathematics \textbf{203} (2001).  

\bibitem{SchThi}
 Thomas Scharf and Jean-Yves Thibon, ``A Hopf-algebra approach to
 inner plethysm,'' \emph{Adv. Math.} \textbf{104} (1994), 30--58.

\bibitem{SchThiWyb}
 Thomas Scharf, Jean-Yves Thibon, and B. G. Wybourne,
 ``Reduced notation, inner plethysm, and the symmetric group,''
 \emph{J. Phys. A} \textbf{24} (1993), 7461--7478.

\bibitem{Thibon}
 Jean-Yves Thibon, ``Hopf algebras of symmetric functions and tensor
 products of symmetric group representations,'' \emph{Internat. J.
 Algebra Comput.} \textbf{1} no. 2 (1991), 207--221.

\bibitem{GSW-code}
 Gregory S. Warrington, Python and SageMath code for creation operators,
 2020.  Available online at 
 \verb-http://www.cems.uvm.edu/~gswarrin/research/research.html-.
\end{thebibliography}
\end{document}